\documentclass[10pt,oneside,a4paper]{amsart}
\usepackage{amsmath,amsfonts,amsthm, amssymb}
\usepackage{array}
\usepackage[all,textures]{xy}
\usepackage{portland}

\theoremstyle{plain}
\newtheorem{theorem}{Theorem}[section]
\newtheorem{proposition}[theorem]{Proposition}
\newtheorem{lemma}[theorem]{Lemma}
\newtheorem{corollary}[theorem]{Corollary}

\theoremstyle{definition}
\newtheorem{definition}[theorem]{Definition}

\theoremstyle{remark}

\newtheorem{remark}[theorem]{Remark}
\newtheorem{example}[theorem]{Example}

\newcommand{\Kern}{\mathrm{Ker}}

\newcommand{\Beeld}{\mathrm{Im}}

\renewcommand{\lim}{\mathrm{lim}}
\newcommand{\colim}{\mathrm{colim}}

\newcommand{\Proj}{\mathrm{Proj}}

\newcommand{\Tor}{\mathrm{Tor}}
\newcommand{\Ext}{\mathrm{Ext}}

\newcommand{\Hom}{\mathrm{Hom}}

\newcommand{\RHom}{\mathrm{RHom}}

\newcommand{\Def}{\mathrm{Def}}

\newcommand{\op}{^{\mathrm{op}}}
\newcommand{\Ob}{\mathrm{Ob}}
\newcommand{\Z}{\mathbb{Z}}

\newcommand{\N}{\mathbb{N}}

\newcommand{\PP}{\mathbb{P}}
\newcommand{\AAA}{\mathfrak{a}}
\newcommand{\BBB}{\mathfrak{b}}

\newcommand{\SSS}{\mathfrak{s}}

\newcommand{\GGGG}{\mathfrak{g}}

\newcommand{\Ab}{\ensuremath{\mathsf{Ab}} }
\newcommand{\Mod}{\ensuremath{\mathsf{Mod}} }

\newcommand{\Gr}{\ensuremath{\mathsf{Gr}} }

\newcommand{\Sh}{\ensuremath{\mathsf{Sh}} }
\newcommand{\mmod}{\ensuremath{\mathsf{mod}} }

\newcommand{\Qch}{\ensuremath{\mathsf{Qch}} }
\newcommand{\Tors}{\ensuremath{\mathsf{Tors}} }
\newcommand{\Qgr}{\ensuremath{\mathsf{Qgr}} }
\newcommand{\Qmod}{\ensuremath{\mathsf{Qmod}} }

\newcommand{\Pro}{\ensuremath{\mathsf{Pro}}}

\newcommand{\Add}{\ensuremath{\mathsf{Add}}}

\newcommand{\lra}{\longrightarrow}
\newcommand{\ra}{\rightarrow}

\newcommand{\bbb}{\ensuremath{\mathcal{B}}}
\newcommand{\ccc}{\ensuremath{\mathcal{C}}}
\newcommand{\ddd}{\ensuremath{\mathcal{D}}}
\newcommand{\eee}{\ensuremath{\mathcal{E}}}

\newcommand{\LLL}{\ensuremath{\mathcal{L}}}

\newcommand{\ooo}{\ensuremath{\mathcal{O}}}

\newcommand{\sss}{\ensuremath{\mathcal{S}}}
\newcommand{\ttt}{\ensuremath{\mathcal{T}}}

\CompileMatrices
\SelectTips{cm}{11}

\title{Abelian and derived deformations in the presence of $\Z$-generating geometric helices}
\author{Olivier De Deken} 
\address[Olivier De Deken]{Departement Wiskunde-Informatica, Middelheimcampus,
Middelheimlaan 1,
2020 Antwerp, Belgium}
\email{olivier.dedeken@myonline.be}
\author{Wendy Lowen} 
\address[Wendy Lowen]{Departement Wiskunde-Informatica, Middelheimcampus,
Middelheimlaan 1,
2020 Antwerp, Belgium}
\email{wendy.lowen@ua.ac.be}
\thanks{The first author is aspirant with the Research Foundation - Flanders (FWO)}
\thanks{The second author is postdoctoral fellow with the Research Foundation - Flanders (FWO)}

\usepackage{fancyhdr}
\begin{document}
\maketitle

\begin{abstract}
For a Grothendieck category $\ccc$ which, via a $\Z$-generating sequence $(\ooo(n))_{n \in \Z}$, is equivalent to the category of ``quasi-coherent modules'' over an associated $\Z$-algebra $\AAA$, we show that under suitable cohomological conditions ``taking quasi-coherent modules'' defines an equivalence between linear deformations of $\AAA$ and abelian deformations of $\ccc$. If $(\ooo(n))_{n \in \Z}$ is at the same time a geometric helix in the derived category, we show that restricting a (deformed) $\Z$-algebra to a ``thread'' of objects defines a further equivalence with linear deformations of the associated matrix algebra. 
\end{abstract}

\section{Introduction}

Deformation theoretic ideas have always been important in non-commutative geometry. Some of the basic non-commutative algebras ``of geometric nature'', like Weyl algebras or quantum planes,  naturally appear as free algebras with ``deformed'' commutativity relations. When we think in terms of affine (non-commutative) geometry, Gerstenhaber's deformation theory of algebras
makes these ideas precise. In this non-commutative affine setup, module categories over non-commutative algebras naturally take over the role of categories of quasi-coherent sheaves, and thanks to homological criteria for geometric notions, these categories harbour a certain geometric side of the picture. In the development of non-commutative projective geometry (see for example \cite{staffordvandenbergh}), a similar story, inspired by Serre's theorem, unfolds. This time, algebraic objects like (non-commutative) graded rings are represented by categories of ``quasi-coherent graded modules'' replacing categories of quasi-coherent sheaves, and these categories are considered to be the primary geometic objects. In fact, since homological algebra really lives on the level of the derived categories, a related point of view goes further and proposes triangulated categories, or rather suitable enhancements thereof, as primary geometric objects.

In the classification of specific non-commutative projective varieties, different types of deformation theoretic arguments have been used. The basic idea is that ``non-commutative deformations of a certain type of commutative space should be non-commutative spaces of that same type''. The question is then: what exactly do we deform? In the different reasonings leading to definitions of, for example, non commutative projective planes, the ``abelian approach'' of \cite{artintatevandenbergh}, \cite{vandenbergh2} and the ``derived approach'' of  \cite{bondalpolishchuk} both eventually lead to the same answer.

In the mean time, a deformation theory 
for abelian categories has been developed in \cite{lowenvandenbergh1, lowenvandenbergh2, lowen2} with as one of the motivations to provide a theoretical framework for some of the ad hoc deformation theoretic arguments in these different approaches.

In this paper, we apply this theory under homological conditions that typically occur for Fano varieties.
The abelian categories we are interested in are categories replacing quasi-coherent sheaves. The most natural framework to define such categories, especially in the deformation context, is that of $\Z$-algebras, i.e. linear categories whose object set is isomorphic to $\Z$.

Since our approach makes use of linear topologies and sheaves, we collect some preliminary results in \S \ref{parlin}. In particular, in Theorem \ref{cor1} we characterize, for a given linear topology $\ttt$ and linear functor $\AAA \lra \ccc$ landing in a Grothendieck category, the situation when $\ccc \cong \Sh(\AAA, \ttt)$, the category of linear sheaves for $\ttt$. This refinement of the main theorem of \cite{lowen1} is a $\ttt$-local version of the characterization of module categories using finitely generated projective generators.

In \S \ref{parparZalg}, we investigate categories of quasi-coherent modules over $\Z$-algebras. For an arbitrary $\Z$-algebra, the category of quasi-coherent modules is defined to be $\Qmod(\AAA) = \Sh(\AAA, \ttt_{\mathrm{tails}})$ for a certain \emph{tails topology} on $\AAA$.  If for a $\Z$-algebra $\AAA$, the category of torsion modules $\Tors(\AAA)$ is localizing, then we have $\Qmod(\AAA) \cong \Mod(\AAA)/\Tors(\AAA)$ (see \S \ref{parqcoh}). In general, the topology $\ttt_{\mathrm{tails}}$ is the ``closure under glueings'' of a covering system $\LLL_{\mathrm{tails}}$ which is very easy to describe: it has the covers $\AAA(-, n)_{\geq m}$ as a basis. If $\Tors(\AAA)$ is localizing, taking the closure under glueings is not necessary, i.e. $\ttt_{\mathrm{tails}} = \LLL_{\mathrm{tails}}$. In particular, this is the case for \emph{finitely generated $\Z$-algebras} which we define in \S \ref{parqcoh}. Although the notion is modeled on finite generation for $\Z$-graded algebras, the term can be deceiving because it actually involves infinitely many generators. For connected, positively graded $\Z$-algebras, finite generation is weaker than the classical noetherian hypothesis on $\AAA$, and even weaker than the ``coherence'' hypothesis introduced in \cite{polishchuk} in order to be able to tackle analytic examples.

In \S \ref{parchar}, we characterize Grothendieck categories $\ccc$ that are equivalent to the category of  quasi-coherent modules over a certain associated $\Z$-algebra. More precisely, we construct $\AAA$ starting from a sequence $(\ooo(n))_{n \in \Z}$ of objects in $\ccc$ by putting
$$\AAA(n, m) = \begin{cases} \ccc(\ooo(-n), \ooo(-m)) & \text{if} \,\, n \geq m \\ 0 & \text{otherwise} \end{cases}$$
If $\ccc \cong \Qmod(\AAA)$, we call $(\ooo(n))_{n \in \Z}$ a \emph{$\Z$-generating} sequence. We now suppose that $\ttt_{\mathrm{tails}} = \LLL_{\mathrm{tails}}$ on $\AAA$. Our characterization in Theorem \ref{cor1} is obtained from Theorem \ref{thechar} by considering the topology $\ttt_{\mathrm{tails}}$. If we restrict our attention to sequences of finitely presented objects in locally finitely presented Grothendieck categories, we recover the familiar geometric condition of ampleness, combined with $\ttt_{\mathrm{tails}}$-projectivity (Corollary \ref{coramp}, see also \cite{polishchuk}). 

Let $(\ooo(n))_{n \in \Z}$ be a $\Z$-generating sequence in a Grothendieck category $\ccc$, with associated $\Z$-algebra $\AAA$. In \S \ref{parabZ}, applying the results of \cite{lowenvandenbergh1} and using the topology $\ttt_{\mathrm{tails}}$, we prove that under the additional assumption that $\Ext^{1,2}_{\ccc}(\ooo(m), X \otimes_k \ooo(n)) = 0$ for $m \leq n$ and $X \in \mmod(k)$, there is an equivalence
$$\Def_{\mathrm{lin}}(\AAA) \lra \Def_{\mathrm{ab}}(\ccc): \BBB \longmapsto \Qmod(\BBB)$$
between linear deformations of $\AAA$ and abelian deformations of $\ccc$ (Theorem \ref{maindef}). 

In \S \ref{pardermat}, we look at the situation in which a $\Z$-generating sequence $(\ooo(n))_{n \in \Z}$ in $\ccc$ is at the same time a geometric helix in the derived category, and investigate the compatibility with deformation (Theorem \ref{helixlift}). Therefore, we necessarily define all the relevant notions, in particular mutations, over an arbitrary commutative ground ring. 
If $(\ooo(n))_{n \in \Z}$ is a geometric $(l,d)$-helix (Definition \ref{defhel}), then $D(\ccc) \cong D(\AAA_{[i - l, i]})$ for every $i$, where $\AAA_{[i - l, i]}$ is the restriction of $\AAA$ to the objects $i-l, \dots i-1, i$. We construct a further equivalence
$$\Def_{\mathrm{lin}}(\AAA) \lra \Def_{\mathrm{lin}}(\AAA_{[i-l, i]}): \BBB \longmapsto \BBB_{[i-l, i]}$$
between linear deformations of $\AAA$ and linear deformations of $\AAA_{[i - l, i]}$
(Theorem \ref{mainderdef}). 

The basic example where these results apply is $\ccc = \Qch(\PP^n)$ with the standard sequence $(\ooo(n))_{n \in \Z}$. Hence, this explains the equivalence of the abelian and the derived approach to non-commutative $\PP^2$'s. It is our intention to apply these results to some concrete geometric helices of sheaves on Fano varieties. This is work in progress. 

\vspace{0,3cm}

\noindent \emph{Acknowledgement.}
The authors are very grateful to Michel Van den Bergh for the ori\-ginal idea of using $\mathbb{Z}$-algebras to capture abelian deformations of projective schemes, and for other interesting ideas.
They also thank Louis de Thanhoffer de V\"olcsey for interesting discussions.

\section{Comparison of linear topologies}\label{parlin}

In \cite{lowen1}, functors $u: \AAA \lra \ccc$ from a small linear category into a Grothendieck category realizing $\ccc$ as a localization of $\Mod(\AAA)$ were characterized intrinsically using linear topologies and sheaves. More precisely, under suitable conditions, a representation $\ccc \cong \Sh(\AAA, \ttt_{\ccc})$ was obtained for a certain topology $\ttt_{\ccc}$ on $\AAA$ (Theorem \ref{GP}). If the functor $u$ is fully faithfull, this is an instance of the Gabriel-Popescu theorem. However, many natural representations occur where this is not the case, and the conditions ``full'' and ``faithful'' have to be replaced by $\ttt_{\ccc}$-local versions. In this section we extend Theorem \ref{GP} to the situation where an additive topology $\ttt$ is specified in advance, and the question is whether $\ccc \cong \Sh(\AAA, \ttt)$. The characterization we obtain in Corollary \ref{cor1} is a $\ttt$-local version of the well known characterization of module categories as having a finitely generated projective generator. This result is applied in Theorem \ref{thechar} to characterize categories of quasi-coherent modules over a $\Z$-algebra.

\subsection{Linear topologies}\label{parGP}

Let $k$ be a commutative groundring.
Let $\AAA$ be a $k$-linear category and $$\Mod(\AAA) = k-\mathsf{Lin}(\AAA^{\op}, \Mod(k)) \cong \Add(\AAA^{\op}, \Ab)$$ the category of right $\AAA$-modules. Then the localizations of $\Mod(\AAA)$ are in 1-1 correspondence with linear topologies on $\AAA$. For a detailed exposition, we refer to \cite{lowen1}. Definitions \ref{def1} and \ref{def2} were made in \cite{lowen4} using different terminology. For the convenience of the reader, we recall the main points. 

A \emph{covering system} on $\AAA$ consists of collections $\ttt(A)$ of subfunctors of $\AAA(-,A) \in \Mod(\AAA)$ for every $A \in \AAA$. The subfunctors $R \in \ttt(A)$ are called \emph{coverings} of $A$. The covering system $\ttt$ is called a \emph{($k$-linear) topology} if the coverings satisfy the $k$-linearized axioms for a Grothendieck topology. In this case, the corresponding category $\Sh(\AAA, \ttt) \subseteq \Mod(\AAA)$ of $k$-linear sheaves defines a localization of $\Mod(\AAA)$. 

\begin{definition}\label{def1}
Let $\ttt$ be a covering system on $\AAA$ and let $f: M \lra N$ be a morphism in $\Mod(\AAA)$. \begin{enumerate}
\item $f$ is a \emph{$\ttt$-epimorphism} if the following holds: for every $y \in N(A)$ there is an $R \in \ttt(A)$ such that $N(g)(y) \in N(A_g)$ is in the image of $f_{A_g}: M(A_g) \lra N(A_g)$ for every $g: A_g \lra A$ in $R$.
\item $f$ is a \emph{$\ttt$-monomorphism} if the following holds: for every $x \in M(A)$ with $f_A(x) = 0 \in N(A)$, there is an $R \in \ttt(A)$ such that $M(g)(x) = 0 \in M(A_g)$ for every $g: A_g \lra A$ in $R$.
\end{enumerate}
\end{definition}

If $\ttt$ is a topology on $\AAA$, we have the following

\begin{lemma}\label{lemtepi}
Let $\ttt$ be a topology on $\AAA$ and $a: \Mod(\AAA) \lra \Sh(\AAA, \ttt)$ the sheafification functor. Let $f: M \lra N$ be a morphism in $\Mod(\AAA)$. 
\begin{enumerate}
\item $f$ is a $\ttt$-epimorphism if and only if $a(f)$ is an epimorphism.
\item $f$ is a $\ttt$-monomorphism if and only if $a(f)$ is a monomorphism.
\end{enumerate}
\end{lemma}

Consider an adjoint pair
$i: \ccc \lra \Mod(\AAA)$
with left adjoint $a: \Mod(\AAA) \lra \ccc$ induced by $u: \AAA \lra \Mod(\AAA) \lra \ccc$.
Let $\ttt_{\ccc}$ be the covering system for which a subfunctor $r: R \subseteq \AAA(-,A)$ is in $\ttt_{\ccc}(A)$ if and only if $a(r)$ is an epimorphism,  in other words if and only if
$$\oplus_{f \in R(A_f)}u(A_f) \lra u(A)$$
is an epimorphism in $\ccc$. We will call such a subfunctor \emph{$\ccc$-epimorphic}.

The fact whether $(a,i)$ is a localization (i.e. $i$ is fully faithful and $a$ is exact) is entirely encoded in the functor $u: \AAA \lra \ccc$. 

\begin{definition}\label{def2}
Consider a functor $u: \AAA \lra \ccc$ as above and a covering system $\ttt$ on $\AAA$.
\begin{enumerate}
\item $u$ is \emph{generating} if the images $u(A)$ for $A \in \AAA$ are a collection of generators for $\ccc$.
\item $u$ is $\ttt$-full if for every $A \in \AAA$ the canonical morphism $\AAA(-,A) \lra \ccc(u(-), u(A))$ is a $\ttt$-epimorphism.
\item $u$ is $\ttt$-faithful if for every $A \in \AAA$ the canonical morphism $\AAA(-,A) \lra \ccc(u(-), u(A))$ is a $\ttt$-monomorphism.
\end{enumerate}
\end{definition}

\begin{theorem}[\cite{lowen1}]\label{GP}
Let $u: \AAA \lra \ccc$ be as above and let 
$$i: \ccc \lra \Mod(\AAA): C \longmapsto \ccc(u(-), C)$$
be the induced functor with left adjoint $a: \Mod(\AAA) \lra \ccc$ extending $u$.
The following are equivalent:
\begin{enumerate}
\item $(a,i)$ is a localization.
\item $u$ is generating,  $\ttt_{\ccc}$-full and $\ttt_{\ccc}$-faithful.
\end{enumerate}
In this situation, $\ttt_{\ccc}$ is a topology on $\AAA$ and $i$ factors through an equivalence $\ccc \cong \Sh(\AAA, \ttt_{\ccc})$.
\end{theorem}

\subsection{A comparison result}\label{parGPref}
Let $\AAA$ be a $k$-linear category, $\ccc$ a $k$-linear Grothendieck category and $u: \AAA \lra \ccc$ a $k$-linear functor.
Let $\ttt$ be a covering system on $\AAA$. In this section we investigate the relation between $\ttt$ and $\ttt_{\ccc}$. In particular, in Corollary \ref{cor1}, we obtain a variant of Theorem \ref{GP} in which, for a given topology $\ttt$ on $\AAA$, we characterize when $u$ gives rise to a localization \emph{with $\ttt_{\ccc} = \ttt$} (\emph{and hence, with} $\ccc \cong \Sh(\ccc, \ttt)$). This characterization is a $\ttt$-local version of the well known characterization of module categories as Grothendieck categories with a set of finitely generated projective generators.

\begin{definition}
Consider $u: \AAA \lra \ccc$ as above and let $\ttt$ be a covering system on $\AAA$.
\begin{enumerate}
\item $u$ is \emph{$\ttt$-projective} if for every $\ccc$-epimorphism $c: X \lra Y$, the morphism $$i(c): \ccc(u(-), X) \lra \ccc(u(-), Y)$$ is a $\ttt$-epimorphism.
\item $u$ is \emph{$\ttt$-finitely presented} if for every filtered colimit $\colim_i X_i$ in $\ccc$ the canonical morphism $$\phi: \colim_i \ccc(u(-), X_i) \lra \ccc(u(-), \colim_i X_i)$$ is a $\ttt$-epimorphism and a $\ttt$-monomorphism.
\item $u$ is  \emph{$\ttt$-ample} if for every $R \in \ttt(A)$, the canonical morphism
$$\oplus_{f \in R(A_f)}u(A_f) \lra u(A)$$ is a $\ccc$-epimorphism.
\end{enumerate}
\end{definition}

\begin{lemma}\label{lemtc}
Consider $u: \AAA \lra \ccc$ as above and suppose that $u$ induces a localization $(i,a)$. Then $u$ is $\ttt_{\ccc}$-projective, $\ttt_{\ccc}$-finitely presented and $\ttt_{\ccc}$-ample.
\end{lemma}

\begin{proof}
For (1), it suffices to note by Lemma \ref{lemtepi} that $a(i(c)) \cong c$ is an epimorphism. For (2), we similarly note that $a(\phi)$ is an isomorphism since $a$ commutes with the filtered colimit. Finally (3) is obvious by definition of $\ttt_{\ccc}$.
\end{proof}

\begin{proposition}\label{comp}
Consider $u: \AAA \lra \ccc$ as above and let  $\ttt$ be a covering system on $\AAA$.
\begin{enumerate}
\item Consider the following:
\begin{enumerate}
\item $\ttt_{\ccc} \subseteq \ttt$
\item $u$ is $\ttt$-full, $\ttt$-faithful, $\ttt$-projective and $\ttt$-finitely presented.
\end{enumerate}
We have:
\begin{enumerate}
\item[(i)] if $u$ induces a localization, then (a) implies (b).
\item[(ii)] if $\ttt$ is a topology, then (b) implies (a).
\end{enumerate}
\item The following are equivalent:
\begin{enumerate}
\item $\ttt \subseteq \ttt_{\ccc}$
\item $u$ is $\ttt$-ample.
\end{enumerate}
\item The following are equivalent:
\begin{enumerate}
\item $u$ induces a localization and $\ttt_{\ccc} = \ttt$.
\item $\ttt$ is a topology and $u$ is generating, $\ttt$-full, $\ttt$-faithful, $\ttt$-projective, $\ttt$-finitely presented and $\ttt$-ample.
\end{enumerate}
\end{enumerate}
\end{proposition}

\begin{proof}
(2) is a tautology by definition of $\ttt_{\ccc}$. (1,i) immediately follows from Lemma \ref{lemtc} and Theorem \ref{GP}. Let us show (1,ii). Consider $R \in \ttt_{\ccc}(A)$, so the canonical morphism
$$c: \oplus_{f \in R)}u(A_f) \lra u(A)$$
ia an epimorphism in $\ccc$. Since $u$ is $\ttt$-projective, the induced $i(c)$ is a $\ttt$-epimorphism. In particular, looking at $1_{u(A)} \in \ccc(u(A), u(A))$, there is a $\ttt$-covering $g: A_g \lra A$ such that for every $g$, we have $u(g) = cb_g$ for some $b_g: u(A_g) \lra \oplus_{f \in R}u(A_f)$. Since $u$ is $\ttt$-finitely presented, there is a $\ttt$-covering $h: B_{gh} \lra A_g$ for every $g$ so that for every $h$ the composition $b_gu(h)$ factors through
$$a_{gh}: u(B_{gh}) \lra \oplus_{f \in R'}u(A_f)$$
for a finite subset $R' \subseteq R$. Writing $p_f$ for the projections of $\oplus_{f \in R'}u(A_f)$, we now have
$$u(gh) = \sum_{f \in R'}u(f)p_fa_{gh}.$$
Since $u$ is $\ttt$-full, we can find for each $f, g, h$ a $\ttt$-covering  $w: W_{ghw} \lra B_{gh}$ for which $p_fa_{gh}u(w) = u(t_{fghw})$.
We thus have
$$u(ghw) = u(\sum_{f \in R'}ft_{fghw}).$$
Finally, since $u$ is $\ttt$-faithful, we can find further $\ttt$-coverings $v: V_{ghwv} \lra W_{ghw}$ for which
$$ghwv = \sum_{f \in R'}ft_{fghw}v,$$ 
whence the maps $ghwv$ belong to the $\ttt_{\ccc}$-covering $R$. Glueing all the $\ttt$-coverings together, we thus find a $\ttt$-covering $T \in \ttt(A)$ with $T \subseteq R$. It follows that $R \in \ttt(A)$, as desired.
\end{proof}

\begin{theorem}\label{cor1}
Consider $u: \AAA \lra \ccc$ as above and let $\ttt$ be a topology on $\AAA$.
The following are equivalent:
\begin{enumerate}
\item $u$ induces a localization and $i: \ccc \lra \Mod(\AAA)$ factors through an equivalence $\ccc \cong \Sh(\AAA, \ttt)$.
\item $u$ is generating, $\ttt$-full, $\ttt$-faithful, $\ttt$-projective, $\ttt$-finitely presented and $\ttt$-ample.
\end{enumerate}
\end{theorem}

\begin{proof}
This immediately follows from Proposition \ref{comp}.
\end{proof}

\begin{remark}
If we take $\ttt = \ttt_{\mathrm{triv}}$ the trivial topology on $\AAA$, for which the only coverings are the representable functors, then in Corollary \ref{cor1} we obtain the well known equivalence between:
\begin{enumerate}
\item $u$ induces an equivalence $\ccc \cong \Mod(\AAA)$.
\item $u$ is the fully faithful inclusion of a set of finitely generated projective generators.
\end{enumerate}
\end{remark}

\begin{remark}
In \cite[Theorem 2.22]{lowen4}, related ideas were used in order to characterize stacks of sheaves over a fibered category on a topological space.
\end{remark}

\section{Quasi-coherent modules over $\Z$-algebras}\label{parparZalg}

$\Z$-algebras were introduced as a convenient generalization of $\Z$-graded algebras (see \cite{bondalpolishchuk}). The category $\Gr(A)$ of graded modules over a $\Z$-graded algebra $A$ is replaced by the category $\Mod(\AAA)$ of modules over a $\Z$-algebra $\AAA$, and most notions can be immediately generalized by using their categorical incarnations. An exception is the notion of finite generation \emph{as a $\Z$-algebra}, which we define in \S \ref{parbas}. 

For geometric applications, one is interested in a quotient category $\Qgr(A)$ of $\Gr(A)$ for a $\Z$-graded algebra $A$, to be considered as the (category of quasicoherent sheaves on the) non-commutative scheme $\Proj(A)$.
The reason for this is Serre's theorem \cite{serre}, and its non-commutative generalization \cite{artinzhang2}. As stated in \cite{staffordvandenbergh}, the Artin-Zhang theorem has an analogue for $\Z$-algebras. Classically, these theorems are formulated in terms of the small abelian categories of finitely generated objects under a noetherian assumption. Motivated by analytic applications  like \cite{polishchukschwarz}, a version of the theorem under weaker ``coherence'' hypotheses  was given in \cite{polishchuk}.

In our approach, we define a category $\Qmod(\AAA)$ of ``quasi-coherent modules'' for a $\Z$-algebra $\AAA$ in complete generality, using a certain topology $\ttt_{\mathrm{tails}}$.
If the category $\Tors(\AAA)$ of torsion modules is localizing, our definition generalizes the classical one. In \S \ref{parqcoh} we investigate some situations in which the topology $\ttt_{\mathrm{tails}}$ has a very easy description. In particular, we show that this is the case for a positively graded, connected, finitely generated $\Z$-algebra, or for a noetherian $\Z$-algebra.

Finally, in \S \ref{parchar}, based on the results of \S \ref{parGPref}, we give a characterization of Grothendieck categories that are equivalent to the category of quasi-coherent modules over a certain associated $\Z$-algebra $\AAA$ (Theorem \ref{thechar}).

\subsection{From $\Z$-graded algebras to $\Z$-algebras}\label{parZalg} By definition, a $\Z$-algebra is simply a $k$-linear category $\AAA$ with an isomorphism $\Ob(\AAA) \cong \Z$. $\Z$-algebras naturally occur when expressing the category $\Gr(A)$ over a $\Z$-graded $k$-algebra $A$ as a module category. Let $A$ be a $\Z$-graded $k$-algebra and let $\Gr(A)$ be the category of $\Z$-graded right $A$-modules. Let $(1)$ be the shift to the left on $\Gr(A)$, $(n) = (1)^n$, and consider the shifted objects $(A(n))_{n \in \Z}$ in $\Gr(A)$. For any $M \in \Gr(A)$, we have
$$\Gr(A)(A(n), M) \cong M_{-n}$$
and consequently the objects $A(n)$ constitute a set of finitely generated projective generators of $\Gr(A)$. Let $\AAA = \AAA(A)$ be the full subcategory of $\Gr(A)$ spanned by the $(A(n))_{n \in \Z}$. Then $\AAA$ becomes a $\Z$-algebra by renaming the object $A(-n)$ by $n$, and we have
$$\AAA(n,m) = \Gr(A)(A(-n), A(-m)) = A_{n - m}.$$
There is an induced equivalence of categories
$$\Gr(A) \cong \Mod(\AAA): M \longmapsto \Gr(A)(A(-?),M) = M_{?}.$$

\subsection{Finitely generated $\Z$-algebras}\label{parbas}
Most definitions are easily given for (modules over) a $\Z$-algebra: they are simply the categorical notions in the category $\Mod(\AAA)$. This holds for example for finitely generated, finitely presented, coherent and noetherian modules, and the associated notions of $\AAA$ being coherent or noetherian. However, it is worthwhile to make some things explicit, in particular in order to obtain a good notion of finite generation of $\AAA$ \emph{as a $\Z$-algebra}.

A $\Z$-algebra $\AAA$ is called \emph{positively graded} if $\AAA(m,n) = 0$ for $m < n$. From now on, we consider a positively graded $\Z$-algebra $\AAA$. 

Concretely, an $\AAA$-module $M$ is given by $k$-modules $(M_n)_{n \in \Z}$ and actions
\begin{equation}\label{action}
M_m \otimes \AAA(n,m) \lra M_n: (x, a) \longmapsto xa
\end{equation}
for $n \geq m$. Consequently, for every $\AAA$-module $M$ and $m \in \Z$, there is a truncated submodule $M_{\geq m}$ of $M$ with
$$(M_{\geq m})_n = \begin{cases} M_n &\text{if}\,\,\, n \geq m; \\ 0 &\text{otherwise.} \end{cases}$$
A corresponding quotient module $M_{< m}$ is defined by
$$0 \lra M_{\geq m} \lra M \lra M_{< m} \lra 0.$$
Of particular interest are the representable modules $\AAA(-,m)$ for $m \in \Z$, whose non-zero values are $\AAA(m,m), \AAA(m+1, m), \dots$. In the case where $\AAA = \AAA(A)$ for a $\Z$-graded algebra $A$, this is precisely the shifted object $A(-m)$.

Although $\AAA$ is not quite an algebra, it can be useful to think in terms of ideals in $\AAA$. A right  ideal $I$ in $\AAA$ is a collection of submodules $I(n,m) \subseteq \AAA(n,m)$ such that for $x \in I(n,m)$ and $a \in \AAA(k,n)$ we have $xa \in I(k,m)$. Left and two sided ideals are defined similarly. Defining a right ideal $I$ in $\AAA$ is equivalent to simultaneously defining submodules
$$I_m = I(-,m) \subseteq \AAA(-,m)$$
of all the representable functors. Examples of right ideals are $\AAA$ itself, and $\AAA_{\geq n}$ defined through the submodules
$$\AAA(-,m)_{\geq n} \subseteq \AAA(-,m).$$
We also have a right ideal $\AAA_+$ defined through the submodules
$$\AAA(-,m)_{\geq m+1} \subseteq \AAA(-,m)$$
which excludes all the pieces $\AAA(m,m)$.
If $M$ is an $\AAA$-module and we are given arbitrary subsets $X_n \subseteq M_m$, and $I$ is a right ideal in $\AAA$, then we can form the submodule
$$XI = \{\sum_{i=0}^kx_ia_i \,\,|\,\, x_i \in X, a_i \in I\} \subseteq M.$$
A module $M$ is \emph{finitely generated} if there exists an epimorphism $\oplus_{i = 1}^k\AAA(-,m_k) \lra M$, or equivalently, if there exist elements $x_1, \dots x_k$ in $M$ with $M = \{x_1, \dots, x_n\}\AAA$.
We will say that a right ideal $I$ in $\AAA$ is \emph{finitely generated} if each of the corresponding submodules $I_m \subseteq \AAA(-,m)$ is finitely generated.

Now we want to formulate what it means for $\AAA$ to be finitely generated as a $\Z$-algebra. To do so, we define the \emph{degree} of an element $a \in \AAA(n,m)$ to be $|a| = n-m$. Hence, the elements of degree $d$ are precisely the elements in $$\bigcup_{n \in \Z}\AAA(n+d, n).$$
We say that a collection of elements $X \subseteq \AAA$ (given by subsets $X(n,m) \subseteq \AAA(n,m)$) \emph{generates} $\AAA$ if every element in $\AAA$ can be written as a (finite) $k$-linear combination of (finite) products of elements in $X$ and elements $1_m \in \AAA(m,m)$. We say that $\AAA$ is \emph{finitely generated} (by $X$) if it is generated by a collection $X$ such that for every $m \in \Z$, the set
$$X_m = \bigcup_{d \in \N}X(m+d,m)$$
is finite. Further, $\AAA$ is called \emph{locally finite} if all the $\AAA(n,m)$ are finitely generated $k$-modules, and \emph{connected} if moreover $\AAA(n,n) = k$ for every $n \in \Z$.

\begin{lemma}
If $\AAA$ is finitely generated and connected, then $\AAA$ is locally finite.
\end{lemma}

\begin{proof}
Suppose $\AAA$ is finitely generated by $X \subseteq \AAA_+$. Consider the $k$-module $\AAA(n,m)$. Then the only elements in $X$ that can appear in a product $a = x_{i_l}\dots x_{i_1} \in \AAA(n,m)$ are elements $x_i \in X(n_i, m_i)$ with $n \geq n_i > m_i \geq m$. Clearly, the total number of such products is finite as soon as every $X(n_i, m_i)$ is finite. \end{proof}

For a positively graded, connected $\Z$-algebra $\AAA$, we have the following characterizations of $\AAA$ being finitely generated:

\begin{proposition}\label{propfin}
Let $\AAA$ be a positively graded, connected $\Z$-algebra. The following are equivalent:
\begin{enumerate}
\item $\AAA$ is finitely generated as a $\Z$-algebra.
\item The ideals $\AAA_{\geq n}$ are finitely generated for all $n \in \Z$, i.e. the modules $\AAA(-,m)_{\geq n}$ are finitely generated for all $n, m \in \Z$.
\item The ideal $\AAA_{+}$ is finitely generated, i.e the modules $\AAA(-,m)_{\geq m+1}$ are finitely generated for all $m \in \Z$.
\end{enumerate}
\end{proposition}

\begin{proof}
Let us first show that (1) implies (2). Suppose $\AAA$ is finitely generated by $X \subseteq \AAA_+$.
The module $\AAA(-,m)_{\geq n}$ has non zero entries $\AAA(n,m), \AAA(n+1, m), \dots$.
Any word $w = x_{i_l}\dots x_{i_1}$ in one of these $k$-modules contains a letter $x \in \AAA(n', m')$ with $m \leq m' < n \leq n'$. Since $\cup_{d \in \N} X(m' +d, m')$ is finite for each of the finitely many $m'$, the total number of such letters $x$ is finite. Now we can write $w = w'xw''$ with $w'x \in \AAA(n',m)$ so $\AAA(-,m)_{\geq n}$ is generated by the words $w'x$. Again since $\AAA$ is finitely generated by $X$, there are only finitely many possibilities for $w'$.
Since (2) trivially implies (3), it remains to show that (3) implies (1). Take for every $m$ a finite generating set $X_m = \{ x_{m_1}, \dots, x_{m_{k_m}}\}$ with $|x_{m_i}| \geq 1$ and $\AAA(-,m)_{\geq m +1} = X_m \AAA$. We claim that $X = \cup_{m \in \Z} X_m$ generates $\AAA$. It is then clear from the definition of $X$ that the generation is finite. Now every $f \in \AAA(n, m)$ with $n > m$ can be written as
$f = \sum_{i = 1}^{k_m} x_{m_i}a_i$ for $a_i \in \AAA(n, l_i)$ with $l_i > m$ and hence $|a_i| < |f|$. The proof is finished by induction on $|f|$. 
\end{proof}

The following shows that \emph{finite} generation of a $\Z$-algebra is a reasonable term, in spite of the fact that it involves an infinite number of generators.

\begin{proposition}\label{remfin}
Let $A$ be a positively graded, connected $\Z$-graded algebra with associated $\Z$-algebra $\AAA$. Then $A$ is finitely generated as an algebra if and only if $\AAA$ is finitely generated as a $\Z$-algebra.
\end{proposition}

\begin{proof}
Suppose $Y = \{ y_1, \dots y_n\}$ with $y_i \in A_{d_i}$, $d_i \geq 1$ is a finite collection of generators for $A$. Put $X_m = \{ x_1^m, \dots, x_n^m\}$ with $x_i^m = y_i \in A_{d_i} = \AAA(m + d_i, m)$. Then $X = \cup_{m \in \Z}X_m$ finitely generates $\AAA$. Suppose conversely that $X = \cup_{m \in \Z} X_m$ finitely generates $\AAA$ and write $X_0 = \{ x_1^0, \dots, x_n^0\}$. We claim that $Y = \{ y_1, \dots y_n\}$ with $y_i = x^0_i \in \AAA(d_i, 0) = A_{d_i}$ generates $A$. To this end we introduce the sets $Y_m = \{ y_1^m, \dots, y_n^m\}$ with $y^m_i = y_i \in A_{d_i} = \AAA(m + d_i, m)$. To prove that $Y$ generates $A$, it is clearly sufficient that $\cup_{m \in \Z} Y_m$ generates $\AAA$. Now consider an element $a \in \AAA(n,m)$. Then the translated element $a' \in \AAA(n - m, 0)$ can be written as a sum of words $x'_{\alpha_1} \dots x'_{\alpha_l}$ with $x'_{\alpha_1} \in X_0 = Y_0$, whence $a$ can be written as a sum of words $x_{\alpha_1} \dots x_{\alpha_l}$ with $x_{\alpha_1} \in Y_m$. The proof is finished by induction on $|a|$.
\end{proof}

By definition, a $\Z$-algebra $\AAA$ is \emph{coherent} if the category $\Mod(\AAA)$ is locally coherent, equivalently if all the representable modules $\AAA(-,n)$ are coherent. In \cite{polishchuk}, this notion is called \emph{weak} coherence and a stronger notion, which we will call \emph{strong coherence}, is considered. Namely, $\AAA$ is strongly coherent if the objects $\AAA(-,n)$ \'and the objects $\AAA(-, n)_{< n+1}$ are coherent.

\begin{proposition}
Let $\AAA$ be a positively graded connected $\Z$-algebra. If the objects $\AAA(-, n)_{< n+1}$ are coherent, then $\AAA$ is finitely generated. In particular, this is the case if 
$\AAA$ is strongly coherent.
\end{proposition} 

\begin{proof}
Consider the exact sequence
$$0 \lra \AAA(-, m)_{\geq m+1} \lra \AAA(-,m) \lra \AAA(-,m)_{< m+1} \lra 0.$$
Since $\AAA(-,m)_{< m+1}$ is coherent and $\AAA(-,m)$ is finitely generated, the kernel $\AAA(-, m)_{\geq m+1}$ is finitely generated. The result follows by Proposition \ref{propfin}.
\end{proof} 

\begin{remark}
There exist finitely generated graded algebras, and hence $\Z$-algebras, that are not coherent (see for example \cite{polishchuk}).
\end{remark}

\subsection{Quasi-coherent modules over a $\Z$-algebra}\label{parqcoh}
In this section we define the category of quasicoherent modules over an arbitrary $\Z$-algebra using the additive ``tails topology''. If the category of torsion modules is localizing, our definition  generalizes the classical one.

A module $M$ over $\AAA$ is called \emph{right bounded} if $M_n = 0$ for $n >> 0$, and \emph{torsion} if it is a directed colimit of right bounded modules. The category of torsion modules is denoted by $\Tors(\AAA)$.

\begin{lemma}\label{opsom}
The following are equivalent for $M \in \Mod(\AAA)$:
\begin{enumerate}
\item $M$ is torsion.
\item For every $x \in M_m$, there is an $n_0 \in \N$ such that for all $n \geq n_0$ and for all $a \in \AAA(n,m)$ we have $0 = xa \in M_n$.
\item $M$ is a union of finitely generated torsion submodules.
\item $M$ is a directed union of finitely generated torsion submodules.
\item Every finitely generated submodule of $M$ is torsion.
\end{enumerate}
Moreover, if $M$ is finitely generated and torsion, then $M$ is right bounded.
\end{lemma}

\begin{proof}
Easy.
\end{proof}

Clearly, the category of right bounded modules is Serre (i.e. closed under subquotients and extensions), and $\Tors(\AAA)$ is closed under coproducts. We are most interested in situations where $\Tors(\AAA)$ is localizing (i.e. closed under subquotiens, extensions and coproducts).

\begin{lemma}\label{astlem}
If all the modules $\AAA(-,m)_{\geq n}$ are finitely generated, then $\Tors(\AAA)$ is a localizing subcategory.
\end{lemma}

\begin{proof}
It is not hard to see that for any $\AAA$, $\Tors(\AAA)$ is closed under subquotients, and it is obviously closed under coproducts. Let us look at an extension
$$\xymatrix{0 \ar[r] & K \ar[r]_f & M \ar[r]_g & Q \ar[r] & 0}$$
in which $K$ and $Q$ are torsion. Take $x \in M_m$. For some $n_0$, we have $g(x)\AAA(-,m)_{\geq n_0} = 0$. Consequently, $K' = x\AAA(-,m)_{\geq n_0}$ is a submodule of $K$. Since $\AAA(-,m)_{\geq n_0}$ is finitely generated, so is $K'$, and consequently $K'$ is right bounded. But then $x\AAA$ is right bounded too, which finishes the proof.
\end{proof}

By Lemma \ref{astlem} and Proposition \ref{propfin}, $\Tors(\AAA)$ is localizing in each of the following situations:
\begin{itemize}
\item $\AAA$ is positively graded, connected and finitely generated.
\item $\AAA$ is noetherian.
\end{itemize}

If $\Tors(\AAA)$ is localizing, we define the category of quasicoherent modules over $\AAA$ to be the quotient category $\Qmod(\AAA) = \Mod(\AAA)/\Tors(\AAA)$. According to \S \ref{parGP}, this category can equivalently be described as a subcategory $\Sh(\AAA, \ttt) \subseteq \Mod(\AAA)$ of linear sheaves. In the corresponding linear topology $\ttt$ on $\AAA$, a submodule $R \subseteq \AAA(-,m)$ is covering if and only if the quotient $\AAA(-,m)/R$ is torsion. 
More precisely, the exact quotient functor
$$\pi: \Mod(\AAA) \lra \Qmod(\AAA)$$
has a fully faithful right adjoint
$$\omega: \Qmod(\AAA) \lra \Mod(\AAA)$$ whose essential image is precisely $\Sh(\AAA, \ttt)$. 
Recall that $\bbb \subseteq \ttt$ is a \emph{basis} for the topology if for every $R \in \ttt$ there exists a $B \in \bbb$ with $B \subseteq R$. If $\bbb$ is a basis for $\ttt$, it is sufficient to check the sheaf property with respect to $\bbb$, and perform sheafification using $\bbb$.

\begin{lemma}
A basis for $\ttt$ is given by the subobjects $\AAA(-,m)_{\geq n} \subseteq \AAA(-,m)$ for $m \leq n \in \Z$.
\end{lemma} 

\begin{proof}
Obviously, $\AAA(-,m)/\AAA(-,m)_{\geq n} = \AAA(-,m)_{<n}$ is right bounded. Now consider an arbitrary subobject $R \subseteq \AAA(-,m)$ for which $\AAA(-,m)/R$ is torsion. Then since $\AAA(-,m)/R$ is finitely generated, it is right bounded. Consequently, $\AAA(-,m)_{\geq n} \subseteq R$ for some $n$.
\end{proof}

For an arbitrary $\Z$-algebra $\AAA$, we define  the covering system $\LLL_{\mathrm{tails}}$ for which $R \in \LLL_{\mathrm{tails}}(m)$ if and only if $\AAA(-,m)_{\geq n} \subseteq R$ for some $m \leq n \in \Z$. Then $\Tors(\AAA)$ is localizing if and only if $\LLL_{\mathrm{tails}}$ defines a topology, and in this situation we have $\Sh(\AAA, \LLL_{\mathrm{tails}}) \cong \Qmod(\AAA)$. We will use this fact to define a category $\Qmod(\AAA)$ in complete generality.
 
\begin{proposition}
Let $\AAA$ be an arbitrary $\Z$-algebra. The covering system $\LLL_{\mathrm{tails}}$ satisfies the identity axiom and the pullback axiom. 
\end{proposition} 

\begin{proof}
Obviously, $\AAA(-,n) = \AAA(-,n)_{\geq n}$ is in $\LLL_{\mathrm{tails}}$. Consider a pullback diagram
$$\xymatrix{ {P'} \ar[r] \ar[d] & {\AAA(-,n)_{\geq m}} \ar[d] \\ P \ar[r] \ar[d] & R \ar[d] \\ {\AAA(-, n')} \ar[r]_{a-} & {\AAA(-,n)}}.$$
Clearly $\AAA(-, n')_{\geq m} \subseteq P'$ which finishes the proof.
\end{proof}

\begin{definition}
Let $\AAA$ be an arbitrary $\Z$-algebra. The \emph{tails topology} $\ttt_{\mathrm{tails}}$ is the smallest topology containing $\LLL_{\mathrm{tails}}$. The \emph{category of quasi-coherent modules over $\AAA$} is by definition $\Qmod(\AAA) = \Sh(\AAA, \ttt_{\mathrm{tails}})$. 
\end{definition}

\begin{remark}
The tails topology $\ttt_{\mathrm{tails}}$ is the intersection of all the topologies containing $\LLL_{\mathrm{tails}}$. It can be obtained from $\LLL_{\mathrm{tails}}$ by ``adding glueings'' of covers in a transfinite induction proces. In order to define sheaves, one only needs the covers in $\LLL_{\mathrm{tails}}$, in other words $\Sh(\AAA, \ttt_{\mathrm{tails}}) = \Sh(\AAA, \LLL_{\mathrm{tails}})$.
\end{remark}

\begin{example}
Consider $A = k[x_1, \dots, x_n, \dots]$, the polynomial ring in countably many variables and let $\AAA$ be the associated $\Z$-algebra with $\AAA(n, m) = A_{n - m}$.
Let $S \subseteq \AAA(-, 0)$ be generated by $\cup_{i = 1}^{\infty} x_i\AAA(-,1)_{\geq i}$. Hence, $S$ contains all monomials of degree $i$ containing $x_i$, but it does not contain $x_{i + 1}^i$. Consequently, $S$ does not contain $\AAA(-, 0)_{\geq i}$ for any $i$, so it is not in $\LLL_{\mathrm{tails}}$. However, if we consider the covering $\AAA(-, 0)_{\geq 1}$, then for all the generators $x_i$, the pullback $x_i^{-1}S$ does contain the covering $\AAA(-,1)_{\geq i}$. It easily follows that arbitrary pullbacks are coverings, so $S$ is ``glued together'' from coverings but fails to be a covering itself. Hence, $\LLL_{\mathrm{tails}}$ fails to be a topology.
\end{example}

\subsection{The characterization}\label{parchar}

Let $\ccc$ be Grothendieck category and let $(\ooo(n))_{n \in \Z}$ be a collection of objects in $\ccc$. We define a $\Z$-algebra $\AAA$ with $\Ob(\AAA) = \Z$ and 
$$\AAA(n,m) = \begin{cases} \ccc(\ooo(-n), \ooo(-m)) & \text{if}\,\, n \geq m \\ 0 & \text{otherwise} \end{cases}$$
so that we obtain a natural functor
$$u: \AAA \lra \ccc: n \longmapsto \ooo(-n).$$

\begin{lemma}\label{lemfffauto}
The functor $u: \AAA \lra \ccc$ is $\ttt_{\mathrm{tails}}$-full and $\ttt_{\mathrm{tails}}$-faithful. 
\end{lemma}

\begin{proof}
The functor $u$ is faithful by construction, whence certainly $\ttt_{\mathrm{tails}}$-faithful. Consider the canonical maps
$$\varphi_{n,m}: \AAA(n,m) \lra \ccc(\ooo(-n), \ooo(-m)).$$
For $n \geq m$, $\varphi_{n,m}$ is an isomorphism by construction and nothing needs to be checked. So take $n < m$ and consider a map $c: \ooo(-n) \lra \ooo(-m)$ in $\ccc$. Consider the $\ttt_{\mathrm{tails}}$-covering $\AAA(-,n)_{\geq m}$. For every $x \in \AAA(k,n)_{\geq m}$, with consequently $k \geq m$, we look at the composition
$$cu(x): \ooo(-k) \lra \ooo(-m).$$
Since $k \geq m$, we have $cu(x)$ in the image of $\varphi_{k,m}$, as desired.
\end{proof}

\begin{definition}\label{defzgen}
If for a collection $(\ooo(n))_{n \in \Z}$ of objects in $\ccc$ the associated functor $u: \AAA \lra \ccc$  induces an equivalence $\ccc \cong \Qmod(\AAA)$, we call $(\ooo(n))_{n \in \Z}$ a \emph{$\Z$-generating} sequence and we call $u$ a \emph{$\Z$-generating} functor.
\end{definition}

\begin{theorem}\label{thechar}
Let $\ccc$ be Grothendieck category, $(\ooo(n))_{n \in \Z}$ a collection of objects in $\ccc$, and $u: \AAA \lra \ccc$ as defined above. Suppose $\LLL_{\mathrm{tails}} = \ttt_{\mathrm{tails}}$ on $\AAA$. The following are equivalent:
\begin{enumerate}
\item $(\ooo(n))_{n \in \Z}$ is a $\Z$-generating sequence in $\ccc$.
\item the following conditions are fulfilled:
\begin{enumerate}
\item the objects $\ooo(n)$ generate $\ccc$, i.e. for every $C \in \ccc$ there is an epimorphism
$$\oplus_{i}\ooo(n_i) \lra C.$$ 
\item $u$ is $\LLL_{\mathrm{tails}}$-ample, i.e. for every $m \leq n$, there is an epimorphism
$$\oplus_{i}\ooo(-n_i) \lra \ooo(-m)$$
with $n_i \geq n$ for every $i$.
\item $u$ is $\LLL_{\mathrm{tails}}$-projective, i.e. for every $\ccc$-epimorphism $c: X \lra Y$ and morphism $f: \ooo(-m) \lra Y$, there is an $n_0 \geq m$ such that every composition $\ooo(-n) \lra \ooo(-m) \lra Y$ with $n \geq n_0$ factors through $c$.
\item $u$ is $\LLL_{\mathrm{tails}}$-finitely presented, i.e. for every filtered colimit $\colim_i X_i$ in $\ccc$ and morphism $f: \ooo(-m) \lra \colim_i X_i$, there is an $n_0 \geq m$ such that for every $n \geq n_0$ every composition $\ooo(-n) \lra \ooo(-m) \lra \colim_i X_i$ factors through $\ooo(-n) \lra \ooo(-m) \lra X_i$ for some $i$. Moreover if a morphism $f: \ooo(-m) \lra X_i$ becomes zero when extended to $\colim_i X_i$, then there is an $n_0 \geq m$ such that for every $n \geq n_0$ every composition $\ooo(-n) \lra \ooo(-m) \lra X_i$ becomes zero when composed with a suitable $X_i \lra X_j$.
\end{enumerate}
\end{enumerate}
\end{theorem}

\begin{proof}
This follows from Theorem \ref{cor1} and Lemma \ref{lemfffauto}.
\end{proof}

When we restrict the situation a bit, we recover the classical geometric notion of ampleness (condition (ab)):

\begin{corollary}\label{coramp}
Let $\ccc$ be a locally finitely presented Grothendieck category, $(\ooo(n))_{n \in \Z}$ a collection of finitely presented objects in $\ccc$, and $u: \AAA \lra \ccc$ as defined above. Suppose $\LLL_{\mathrm{tails}} = \ttt_{\mathrm{tails}}$ on $\AAA$. The following are equivalent:
\begin{enumerate}
\item $(\ooo(n))_{n \in \Z}$ is a $\Z$-generating sequence in $\ccc$.
\item the following conditions are fulfilled:
\begin{enumerate}
\item[(ab)]  $(\ooo(n))_{n \in \Z}$ is ample, i.e. for every finitely presented object $C \in \ccc$, there is an $n_0$ such that for every $n \geq n_0$, there is an epimorphism
$$\oplus_i \ooo(-n_i) \lra C$$
with $n_i \geq n$ for every $i$.
\item[(c)] $u$ is $\LLL_{\mathrm{tails}}$-projective, i.e. for every $\ccc$-epimorphism $c: X \lra Y$ and morphism $f: \ooo(-m) \lra Y$, there is an $n_0 \geq m$ such that every composition $\ooo(-n) \lra \ooo(-m) \lra Y$ with $n \geq n_0$ factors through $c$.
\end{enumerate}
\end{enumerate}
\end{corollary}

\begin{proof}
Since the objects $\ooo(n)$ are finitely presented, condition (d) in Theorem \ref{thechar} is automatically fulfilled. It suffices to show the equivalence of (ab) and (a)$\wedge$(b). First, suppose (a) and (b) hold and take a finitely presented $C$. By (a), there is an epimorphism $\oplus_i \ooo(-n_i) \lra C$ and we may suppose that the number of $n_i$'s is finite. Put $n_0 = \max\{n_i\}$ and take $n \geq n_0$. Since $n_i \leq n$ for all $i$, by (b) we get an epimorphism $\oplus_j\ooo(-n_{ij}) \lra \ooo(-n_i)$ for every $i$ with $n_{ij} \geq n$ for all $j$. Consequently, we get an epimorphism $\oplus_{i,j} \ooo(-n_{ij}) \lra C$ with $n_{ij} \geq n$ for all $i, j$. Conversely, suppose (ab) holds. For (b), put $C = \ooo(-m)$ and let $n_0$ be as in (ab). For a given $m \leq n$, put $n' = \max\{ n_0, n\}$. Then (ab) yields an epimorphism $\oplus_i \ooo(-n_i) \lra \ooo(-m)$ with $n_i \geq n' \geq n$ for every $i$. For (a), take an arbitrary $C \in \ccc$. There is a set of finitely presented generators $C_i$ with an epimorphism $\oplus_i C_i \lra C$. Then by (ab), we can take further epimorphisms $\oplus_j\ooo(-n_{ij}) \lra C_i$ to finish the proof. 
\end{proof}

\section{Abelian deformations and $\Z$-algebras}\label{parabZ}

Let $(\ooo(n))_{n \in \Z}$ be a $\Z$-generating sequence in a Grothendieck category $\ccc$, and let $\AAA$ be the associated $\Z$-algebra. Using \cite{lowenvandenbergh1}, we show that, under the additional assumption that $\Ext^{1,2}_{\ccc}(\ooo(m), X \otimes_k \ooo(n)) = 0$ for $m \leq n$ and $X \in \mmod(k)$, ``taking quasi-coherent modules'' defines an equivalence between linear deformations of $\AAA$ and abelian deformations of $\ccc$ (Theorem \ref{maindef}).

\subsection{Abelian deformations} In \cite{lowenvandenbergh1, lowen2}, a deformation theory of abelian categories was established with as one of the motivations to provide a theoretical framework for some of the ad hoc deformation theoretic arguments in \cite{bondalpolishchuk} and \cite{vandenbergh2}. Let us recall the main points. 

First, we need some notions for a $k$-linear abelian category $\ccc$, where $k$ is a coherent commutative ground ring. We have natural actions $\Hom_R(-,-): \mmod(R) \otimes \ddd \lra \ddd$ and $- \otimes_R -: \mmod(R) \otimes \ddd \lra \ddd$. We call an object $C \in \ccc$ \emph{flat} if $- \otimes_k C: \mmod(k) \lra \ccc$ is exact and we call $C$ \emph{coflat} if $\Hom_k(-,C): \mmod(k) \lra \ccc$ is exact. To obtain a good deformation theory, we use an intrinsic notion of flatness (\cite{lowenvandenbergh1}) for abelian categories, which is such that a $k$-linear category $\AAA$ is flat (in the sense of having $k$-flat hom-modules) if and only if the abelian category $\Mod(\AAA)$ is flat in the new abelian sense. 

Throughout, $R \lra k$ is a surjection between coherent, commutative rings such that $k$ is finitely presented over $R$ and the kernel $I = \Kern(R \lra k)$ is nilpotent.
Let $\ddd$ be an abelian $R$-linear category. We put
$$\ddd_k = \{ D \in \ddd \,\, |\,\, ID = \Beeld(I \otimes_R D \lra D) = 0\} \subseteq \ddd.$$
For a flat $k$-linear abelian category $\ccc$, an \emph{abelian $R$-deformation} is a flat $R$-linear abelian category $\ddd$ with an equivalence $\ccc \lra \ddd_k$.  Thus, an abelian category $\ccc$ ``sits inside'' its deformations $\ddd$, and the inclusion map $\ccc \lra \ddd$ has the functors $k \otimes_R -$ as a left adjoint and $\Hom_R(k, -)$ as a right adjoint.

In contrast, for a flat $k$-linear category $\AAA$, a \emph{linear $R$-deformation} is a flat $R$-linear category $\BBB$ with an equivalence $k \otimes_R \BBB \lra \AAA$ (where the tensor product is taken hom-module by hom-module, the object set remaining fixed).

We have the following basic result, relating the resulting abelian deformation theory to Gerstenhaber's deformation theory of algebras:

\begin{proposition}\cite{lowenvandenbergh1}
For a flat $k$-linear category $\AAA$, there is an equivalence of deformation functors
\begin{equation}\label{eqbaseq}
\Def_{\mathrm{lin}}(\AAA) \lra \Def_{\mathrm{ab}}(\Mod(\AAA)): \BBB \longmapsto \Mod(\BBB).
\end{equation}
Here, $\Def_{\mathrm{lin}}$ stands for linear deformations and $\Def_{\mathrm{ab}}$ stands for abelian deformations.
\end{proposition}

From now on, when speaking about deformations, the suitable flatness hypothesis will always be implicitly understood.

\subsection{Deformation and localization}
The relation between deformation and localization is summarized in the following

\begin{theorem}\cite{lowenvandenbergh1} \label{locdefthm0}
Let $\ccc \subset \ddd$ be an abelian $R$-deformation. Then the maps
$$\SSS(\ddd) \lra \SSS(\ccc): \sss \longmapsto \sss \cap \ccc$$
and $$\SSS(\ccc) \lra \SSS(\ddd): \sss \longmapsto \langle \sss \rangle_{\ddd} = \{ D \in \ddd \,\, |\,\, k \otimes_R D \in \sss\}$$
are inverse bijections between the sets of Serre subcategories in $\ccc$ and $\ddd$.
If $\ccc$ is Grothendieck, they restrict to inverse bijections between the sets of localizing subcategories. 

For a given localizing $\LLL$ in $\SSS(\ccc)$, the quotient $\ddd/\langle \LLL\rangle_{\ddd}$ is a deformation of $\ccc/\LLL$. We thus obtain a map
\begin{equation}\label{eqlocmap}
\Def_{\mathrm{ab}}(\ccc) \lra \Def_{\mathrm{ab}}(\ccc/\LLL).
\end{equation}
\end{theorem}

Let $\varphi: \BBB \lra \AAA$ be an $R$-deformation of a $k$ linear category $\AAA$. Let $k \otimes_R -: \Mod(\BBB) \lra \Mod(\AAA)$ denote the left adjoint of the corresponding abelian deformation $\Mod(\AAA) \subseteq \Mod(\BBB)$.
We will now translate the bijections between localizing subcategories of Theorem \ref{locdefthm0} in terms of bijections between topologies.
We first define maps between covering systems
$$\mathrm{cov}(\BBB) \lra \mathrm{cov}(\AAA): \ttt \longmapsto \varphi(\ttt)$$
and
$$\mathrm{cov}(\AAA) \lra \mathrm{cov}(\BBB): \ttt \longmapsto \varphi^{-1}(\ttt).$$
For a subfunctor $S \subseteq \BBB(-,B)$, we define $\varphi(S)$ as the subfunctor of $\AAA(-, f(B))$ containing precisely the maps $\varphi(f)$ for $f: B_f \lra B$ in $S$. Alternatively, $\varphi(S)$ is the image of $k \otimes_R(S \ra \BBB(-,B))$.
Now we put
$$\varphi(\ttt) = \{ \varphi(S) \,\, |\,\, S \in \ttt\}$$
and
$$\varphi^{-1}(\ttt) = \{ S \,\, |\,\, \varphi(S) \in \ttt\}.$$

\begin{proposition}\cite{lowen4} \label{topolift}
The maps we just defined restrict to bijections between topologies that fit into commutative squares
$$\xymatrix{ {\mathrm{top}(\BBB)} \ar[d] \ar[r] & {\LLL(\Mod(\BBB))} \ar[d] \\
{\mathrm{top}(\AAA)} \ar[r] & {\LLL(\Mod(\AAA))}}$$
in which the horizontal bijections are the standard ones.
\end{proposition}

\begin{proof}
Let $\ttt$ be a topology on $\AAA$. All we have to do is determine the corresponding topology on $\BBB$ by going first to the right (obtaining $\sss$), then up (obtaining $\langle \sss \rangle_{\Mod(\BBB)}$) and then to the left (yielding $\ttt'$) in the diagram. A subfunctor $T \subseteq \BBB(-,B)$ is in $\ttt'$ if and only if the quotient $\BBB(-,B)/T$ in $\langle \sss \rangle_{\Mod(\BBB)}$. From the exact sequence 
$$k \otimes_R T \lra \AAA(-,A) \lra k \otimes_R (\BBB(-,B)/T) \lra 0$$
and the definition of $\langle \sss \rangle_{\Mod(\BBB)}$ we deduce that this is equivalent to $\varphi(T) \in \ttt$. To construct the inverse bijection, first note that every subfunctor $T \subseteq \AAA(-,A)$ can be written as $\varphi(P)$ where $P$ is the pullback of $T$ along $\BBB(-,\bar{A}) \lra \AAA(-,A)$ for an arbitrary lift $\bar{A}$ of $A$. For a topology $\ttt$ on $\bbb$, we are looking for a topology $\ttt'$ on $\AAA$ with $\varphi^{-1}(\ttt') = \ttt$. Obviously $\ttt'$ has to contain all the subfunctors $\varphi(S)$ for $S$ in $\ttt$. By the previous remark, this is all it can contain.
\end{proof}

\subsection{Deformations of $\Z$-algebras}
Let $\AAA$ be a $\Z$-algebra. Consider the canonical map
$$\lambda: \Def_{\mathrm{lin}}(\AAA) \lra \Def_{\mathrm{ab}}(\Qmod(\AAA))$$
which is the composition of \eqref{eqbaseq} and \eqref{eqlocmap}.
Next we show that it has the desirable prescription.

\begin{proposition}
The canonical map $\lambda$ is given by $\lambda(\BBB) = \Qmod(\BBB)$.
\end{proposition}

\begin{proof}
Consider $\varphi: \BBB \lra \AAA$. By Proposition \ref{topolift}, it suffices to show that 
$$\varphi(\ttt_{\mathrm{tails}, \BBB}) = \ttt_{\mathrm{tails}, \AAA}.$$
For the basic coverings $\BBB(-,m)_{\geq n}$, it is clear that since $\varphi$ is full, we have $\varphi(\BBB(-,m)_{\geq n} = \AAA(-,m)_{\geq n}$, and consequently 
$\varphi(\LLL_{\mathrm{tails}, \BBB}) \subseteq \LLL_{\mathrm{tails}, \AAA}.$
Furthermore, for an arbitrary subfuncor $\AAA(-,m)_{\geq n} \subseteq T \subseteq \AAA(-,m)$, we consider the pullbacks
$P$ of $T$ and $P'$ of $\AAA(-,m)_{\geq n}$ along $\BBB(-,m) \lra \AAA(-,m)$. Then clearly $P' = \BBB(-,m)_{\geq n}$ and $\varphi(P) = T$. This already shows that 
$$\varphi(\LLL_{\mathrm{tails}, \BBB}) = \LLL_{\mathrm{tails}, \AAA}.$$
Consider the topology $\varphi^{-1}\ttt_{\mathrm{tails},\AAA}$ on $\BBB$ which corresponds to $\ttt_{\mathrm{tails},\AAA}$ under the bijection of Proposition \ref{topolift}. Since $\LLL_{\mathrm{tails}, \BBB} \subseteq \varphi^{-1}\ttt_{\mathrm{tails},\AAA}$ we have $\ttt_{\mathrm{tails}, \BBB} \subseteq \varphi^{-1}\ttt_{\mathrm{tails},\AAA}$. After taking $\varphi$, it follows that $\LLL_{\mathrm{tails}, \AAA} \subseteq \varphi \ttt_{\mathrm{tails}, \BBB} \subseteq \ttt_{\mathrm{tails}, \AAA}$ and hence $\varphi \ttt_{\mathrm{tails}, \BBB} = \ttt_{\mathrm{tails}, \AAA}$ as desired.
\end{proof}

In the next theorem, we give conditions under which $\lambda$ is an equivalence.

\begin{theorem}\label{maindef}
Let $\ccc$ be a Grothendieck category with a $\Z$-generating sequence $(\ooo(n))_{n \in \Z}$ and associated $\Z$-generating functor $\AAA \lra \ccc: n \longmapsto \ooo(-n)$. 
Suppose the objects $\ooo(n)$ are flat and suppose for $m \leq n$, $i = 1, 2$ and $X \in \mmod(k)$ we have
$$\Ext^i_{\ccc}(\ooo(m), X \otimes_k \ooo(n)) = 0.$$
Then $$\lambda: \Def_{\mathrm{lin}}(\AAA) \lra \Def_{\mathrm{ab}}(\ccc): \BBB \longmapsto \Qmod(\BBB)$$
is an equivalence of deformation functors. More precisely, for every deformation $\ddd$ of $\ccc$ there is a linear deformation $\BBB$ of $\AAA$ and a functor $\BBB \lra \ddd$ satisfying the same conditions as $\AAA \lra \ccc$.
\end{theorem}

\begin{proof}
This is an application of \cite[Thm. 8.14]{lowenvandenbergh1}. Clearly, the relation $n \geq m$ on $\Ob(\AAA)$ satisfies the requirement in \cite[Prop. 8.12]{lowenvandenbergh1} that $n \not \geq m$ implies that $\AAA(n, m) = 0$, and $n \geq m$ implies that $\AAA(n, m) \lra \ccc(\ooo(-n), \ooo(-m))$ is an isomorphism, by construction of $\AAA$. 

Let us sketch the construction of the inverse equivalence to $\lambda$ for further use. Consider an abelian $R$-deformation $\ccc \lra \ddd$ along with its left adjoint $k \otimes_R -: \ddd \lra \ccc$. Since $\Ext^{1,2}_{\ccc}(\ooo(n), I \otimes_k \ooo(n)) = 0$ (where $I = \Kern(R \lra k)$), the objects $\ooo(n)$ have unique flat  lifts $\bar{\ooo}(n) \in \ddd$ along $k \otimes_R -$  (see \cite{lowen2}). We then build up a linear category $\BBB$ with a functor $\BBB \lra \ddd$ following the same principles of $\AAA \lra \ccc$: we put 
$$\BBB(n,m) = \begin{cases} \ddd(\bar{\ooo}(-n), \bar{\ooo}(-m)) & \text{if}\,\, n \geq m \\ 0 & \text{otherwise} \end{cases}$$
The conditions on $\AAA$ are used to prove that $\BBB$ is a linear deformation of $\AAA$, and that we thus obtain a map $\rho: \Def_{\mathrm{ab}}(\ccc) \lra \Def_{\mathrm{lin}}(\AAA)$
inverse to $\lambda$.
\end{proof}

\subsection{Finiteness conditions}
Let $\AAA$ be a $\Z$-algebra. According to \S \ref{parqcoh}, if $\AAA$ is noetherian or positively graded, connected and finitely generated, then $\LLL_{\mathrm{tails}} = \ttt_{\mathrm{tails}}$. 
Although this equality does not lift under deformation, the individual finiteness conditions do.

\begin{proposition}
Let $\BBB$ be an $R$-linear deformation of $\AAA$. The following conditions lift from $\AAA$ to $\BBB$:
\begin{enumerate}
\item $\AAA$ is connected.
\item $\AAA$ is positively graded.
\item $\AAA$ is locally finite.
\item $\AAA$ is finitely generated.
\item $\AAA$ is noetherian.
\end{enumerate}
\end{proposition}

\begin{proof}
(1) If the flat $R$-module $\BBB(n, n)$ satisfies $k \otimes_R \BBB(n, n) = k$, then necessarily $\BBB(n, n) = R$. (2) If $k \otimes_R \BBB(m, n) = 0$, then $I\BBB(m, n) \cong \BBB(m, n) = 0$.  (3) Follows from the exact sequences $0 \lra I \otimes_k \AAA(m, n) \lra \BBB(m, n) \lra \AAA(m, n) \lra 0$ since $I$ is finitely generated. (4) Consider the abelian deformation $\Mod(\BBB)$ of $\Mod(\AAA)$. For every $\BBB$-module we have an exact sequence $0 \lra IM \lra M \lra k \otimes_R M \lra 0$ where $IM$ is the image of $I \otimes_k (k \otimes_R M) = I \otimes_R M \lra M$. If $k \otimes_R M$ is a finitely generated (resp. noetherian) $\AAA$-module, then so is $IM$ and they are both finitely generated (resp. noetherian) $\BBB$-modules. It follows that $M$ is too. For (4), it suffices to apply the statement about finite generation to $M = \BBB(-, n)_{\geq m}$ with $k \otimes_R \BBB(-, n)_{\geq m} = \AAA(-, n)_{\geq m}$. For (5) we apply the statement about noetherian modules to $M = \BBB(-, n)$.
\end{proof}

\section{Derived deformations and matrix algebras}\label{pardermat}

Let $\ccc$ be a Grothendieck category. In this section we look at a $\Z$-generating sequence $(\ooo(n))_{n \in \Z}$ which is at the same time a geometric helix in the derived category (Definition \ref{defhel}), and we investigate this situation under deformation. If $\AAA$ is the $\Z$-algebra associated to the sequence, and $\AAA_{[i - k, i]}$ the restriction to the objects $i-k, \dots, i -1, i$ corresponding to a thread $\ooo(-i), \ooo(-i+1), \dots \ooo(-i + k)$ of the helix, then we prove that ``restriction to these objects'' defines an equivalence between linear deformations of $\AAA$ and of $\AAA_{[i - k, i]}$ (Theorem \ref{mainderdef}).

\subsection{Derived actions}

Let $\ccc$ be a \emph{flat} $k$-linear Grothendieck category. Consider the natural actions $\Hom_k(-,-): \mmod(k) \otimes \ccc \lra \ccc: (X, C) \longmapsto \Hom_k(X,C)$ and $- \otimes_k -: \mmod(k) \otimes \ccc \lra \ccc: (X,C) \longmapsto X \otimes_k C$.

\begin{proposition} \label{balanced} These actions extend to balanced derived actions 
$$\RHom_k(-, -) : D^-(\mmod(k)) \otimes D^+(\ccc) \lra D^+(\ccc)$$ and
$$- \otimes^L -: D^-(\mmod(k)) \otimes D^-(\ccc) \lra D^-(\ccc).$$
\end{proposition}

\begin{proof}
These are classical balancedness arguments. Since $\ccc$ has enough injectives, the first one is somewhat easier. Let us look at the second one. Here, after an enlargement of universe, we first construct $D^-(\ccc)$ using the Pro-completion $\Pro(\ccc)$. This new $k$-linear category has a natural action
$- \otimes_k -: \mmod(k) \otimes \Pro(\ccc) \lra \Pro(\ccc)$ which is easily seen to be the Pro-extension of the original action, i.e. $M \otimes_k \lim_i C_i = \lim_i (M \otimes_k C_i)$.  
Now we consider $D^-(\Pro(\ccc)$, with $D^-(\ccc) = D^-_{\ccc}(\Pro(\ccc))$. By \cite{lowenvandenbergh1}, $\Pro(\ccc)$ is again flat, whence projectives in $\Pro(\ccc)$ are flat objects. For $M \in D^-(\mmod(k))$ and $A \in D^-(\Pro(\ccc))$, take projective resolutions $P_M \lra M$ and $P_A \lra A$. Since $P^i_M$ is a summand of a finite free module, $P_M^i \otimes_k -$ is an exact functor. By flatness of $\Pro(\ccc)$, the functors $- \otimes_k P^i_A$ are exact as well. Now
$$P_M \otimes_k A \cong P_M \otimes_k P_A \cong M \otimes P_A$$
follows from the classical bicomplex argument. Finally, note that for $C \in D^-(\ccc)$, $M \otimes^L C \cong P_M \otimes C \in D^-(\ccc)$.
\end{proof}

\begin{proposition}
For $C \in \ccc$,
\begin{enumerate}
\item $C$ is flat if and only if  $M \otimes_k^L C \cong M \otimes_k C$ for every $M \in D^{-}(\mmod(k))$;
\item $C$ is coflat if and only if  $\RHom_k(M, C) \cong \Hom_k(M,C)$ for every $M \in D^{-}(\mmod(k))$.
\end{enumerate}
\end{proposition}

For an arbitrary complex $D \in C(\ccc)$, we define $\RHom_k(-, D)$ and $- \otimes^L D$ as the derived functors in the first argument $X \in D^-(\mmod(k))$. The resulting functor is well-defined on $D(\ccc)$ as well.

\begin{proposition}\label{propchange}
For $X \in D^-(\mmod(k))$, $C \in D^b(\ccc)$ and $D \in D(\ccc)$, we have:
$$\RHom_{\ccc}(C, \RHom_k(X,D)) \cong \RHom_k(X, \RHom_{\ccc}(C,D)) \cong \RHom_{\ccc}(X \otimes^L_k C, D)$$
in $D(\Mod(k))$.
\end{proposition}

\begin{proof}
We may suppose that $X$ is a bounded above complex of finitely generated free $k$-modules, and that $D$ is homotopy injective. Then
\begin{align*}
\RHom_k(X, \RHom_{\ccc}(C,D)) & = \Hom_k(X, \Hom_{\ccc}(C,D))\\ & =  \Hom_{\ccc}(C, \Hom_k(X, D))\\
& = \RHom_{\ccc}(C, \Hom_k(X,D))\\
& = \RHom_{\ccc}(C, \RHom_k(X,D))
\end{align*}
where we have used Lemma \ref{lemhopyinj} in the third step.
The isomorphism between the first and the last expression is similar.
\end{proof}

\begin{lemma}\label{lemhopyinj}
Let $X$ be a bounded above complex of finite projective $k$-modules, and $D$ a homotopy injective complex of $\ddd$-objects. Then $\Hom_k(X,D)$ is homotopy injective. 
\end{lemma}

\begin{proof}
For an acyclic complex $E$ of $\ddd$-objects, we have $\Hom_{\ddd}(E, \Hom_k(X,D)) = \Hom_k(X, \Hom_{\ddd}(E,D))$. 
\end{proof}

For some applications, it will be useful to extend the actions from $\mmod(k)$ to $\Mod(k)$. This is possible since $\ccc$ is a complete and cocomplete category. For example, for $C \in \ccc$, we define $- \otimes_k C: \Mod(k) \lra \ccc$ as the unique colimit preserving functor with $k \otimes_k C = C$. For $C \in C(\ccc)$, we obtain a derived functor $- \otimes^L_k C$ in the first argument. It is easily seen that $- \otimes^L_k C$ preserves coproducts.

\begin{proposition}\label{propcompchange}
For $X \in D^-(\Mod(k))$, $C, D \in C(\ccc)$ with $C$ compact, we have:
$$X \otimes^L_k \RHom_{\ccc}(C,D) \cong \RHom_{\ccc}(C, X \otimes^L_k D)$$
in $D(\Mod(k))$.
\end{proposition}

\begin{proof}
Clearly, the isomorphism holds for $X = k$. Now every $X \in  D^-(\Mod(k))$ can be obtained from $k$ using cones, shifts and coproducts. By definition, both sides of the isomorphism define triangulated functors in $X$. It then suffices to show that both these functors preserve arbitrary coproducts. This easily follows using compactness of $C$.
\end{proof}

Finally, for a fixed $X \in C^-(\mmod(k))$, we will need the derived functors
$$\RHom_R^{I \! I}(X, -)): D(\ccc) \lra D(\ccc)$$
and

$${X} ^{I\!I}\! \!\! \otimes^L_R -: D(\Pro(\ccc)) \lra D(\Pro(\ccc))$$
defined using homotopy injective resolutions in $\ccc$ and homotopy projective resolutions in $\Pro(\ccc)$ respectively. 
According to Proposition \ref{balanced}, these functors coincide with $\RHom_R(X,-)$ and $X\otimes_R -$ on $D^+(\ccc)$ and $D^-(\ccc)$ respectively.

\subsection{Deformations}
Let $\ccc \lra \ddd$ be an abelian deformation of (flat) Grothendieck categories with adjoints $k \otimes_R -$ and $\Hom_R(k,-)$. 
We obtain the derived functors
$$\RHom_R^{I \! I}({k}, D)): D(\ddd) \lra D(\ccc)$$
and
$${k}\,^{I\!I}\! \!\! \otimes^L_R -: D(\Pro(\ddd)) \lra D(\Pro(\ccc))$$
which are right and left adjoint to the functors $D(\ccc) \lra D(\ddd)$ and $D(\Pro(\ccc)) \lra D(\Pro(\ddd))$ respectively.

\begin{proposition}[Derived change of rings]\label{CR}
Consider $D \in D(\ddd)$ and $X \in C^-(\mmod(k))$. We have
\begin{enumerate}
\item $\RHom^{I \! I}_k({X}, \RHom_R^{I \! I}({k}, D)) = \RHom_R^{I \! I}({X}, D)$;
\item ${X} ^{I\! I}\!\!\!\otimes^L_k ({k}\, ^{I \! I}\!\!\!\otimes^L_R D) = {X} ^{I \! I}\!\!\!\otimes^L_R D$.
\end{enumerate}
\end{proposition}

\begin{proof}
It suffices to note that $\Hom_R(k, -)$ maps a homotopy injective complex of $\ddd$-objects to a homotopy injective complex of $\ccc$-objects, and that $k \otimes_R -$ maps a homotopy projective complex of $\Pro(\ddd)$-objects to a homotopy-projective complex of $\Pro(\ccc)$-objects.
\end{proof}

\begin{proposition}[Derived Nakayama]
If $D \in D(\ddd)$ satisfies $\RHom_R^{I \! I}({k}, D) = 0$ or ${k}\, ^{I \! I}\!\!\!\otimes_R^L D = 0$, then $D = 0$
\end{proposition}

\begin{proof}
We have a triangle $\RHom_R^{I \! I}({k}, D) \lra D \lra \RHom_R^{I \! I}({I},D) \lra $. Furthermore, by Proposition \ref{CR}, $\RHom_R^{I \! I}({I},D) = \RHom_k^{I \! I}({I}, \RHom_R^{I \! I}({k},D)) = 0$.
\end{proof}

\begin{proposition}\label{propcomplift}
Consider $G \in D(\ddd)$ such that ${k}\, ^{I \! I}\!\! \!\otimes_R^L G$ is compact in $D(\ccc)$. Then $G$ is compact in $D(\ddd)$.
\end{proposition}

\begin{proof}
For a collection $D_{\alpha}$ in $D(\ddd)$, consider the canonical morphism
$$\bigoplus_{\alpha}\RHom_{\ddd}(G, D_{\alpha}) \lra \RHom_{\ddd}(G, \bigoplus_{\alpha} D_{\alpha}).$$
We are to show that this is a quasi-isomorphism.
For the collection $\RHom_R^{I \! I}({I}, D_{\alpha})$, and similarly for $\RHom_R^{I \! I}({k}, D_{\alpha})$, the corresponding map can be rewritten as:
$$\bigoplus_{\alpha}\RHom_{\ccc}({k}\, ^{I\! I}\!\!\!\otimes_R^LG, \RHom_R^{I \! I}({I},D_{\alpha})) \lra \RHom_{\ccc}({k}\, ^{I \! I}\!\!\!\otimes^L G, \bigoplus_{\alpha} \RHom_R^{I \! I}({I},D_{\alpha})),$$
which is a quasi-isomorphism by compactness of ${k}\, ^{I \! I}\!\!\!\otimes_R^L G$ in $\ccc$.
From the triangles
$$\triangle_{\alpha} = \RHom_R^{I \! I}({k}, D_{\alpha}) \lra D_{\alpha} \lra \RHom_R^{I \! I}({I}, D_{\alpha}) \lra$$
we obtain a morphism of triangles
$$\bigoplus_{\alpha}\RHom_{\ddd}(G, \triangle_{\alpha}) \lra \RHom_{\ddd}(G, \bigoplus_{\alpha} \triangle_{\alpha})$$ whence the result follows.
\end{proof}

\begin{proposition}\label{propgenlift}
Consider a collection $\GGGG$ of objects of $D^-(\ddd)$ such that the collection ${k}\, ^{I \! I}\!\!\!\otimes^L_R \GGGG = \{{k}\, ^{I\!I}\!\!\!\otimes^L_R G \,\, |\,\, G \in \GGGG\}$ compactly generates $D(\ccc)$.
Then $\GGGG$ compactly generates $D(\ddd)$.
\end{proposition}

\begin{proof}
Since ${k}\, ^{I \! I}\!\!\!\otimes^L_R \GGGG$ compactly generates $D(\ccc)$, $D(\ccc)$ is the smallest triangulated subcategory of $D(\ccc)$ which is closed under coproducts and contains ${k} ^{I \! I}\!\!\otimes^L_R \GGGG$. Now for every object $D \in D(\ddd)$, the triangle $\RHom_R^{I \! I}({I}, D) \lra D \lra \RHom_R^{I\! I}({k}, D) \lra$ shows that $D(\ddd)$ is the smallest triangulated subcategory of $D(\ddd)$ containing ${k}\, ^{I \! I}\!\!\otimes^L_R \GGGG$. 
By Proposition \ref{propcomplift}, the objects in $\GGGG$ are compact in $D(\ddd)$. The proof will be finished if we can show that the objects in ${k}\, ^{I \! I}\!\!\!\otimes^L_R \GGGG$ are in the smallest triangulated subcategory of $D(\ddd)$ closed under coproducts and containing $\GGGG$. To see this, note that for $G \in D^-(\ddd)$, ${k}\, ^{I \! I}\!\!\!\otimes^L_R G = k \otimes^L_R G$ is computed using a resolution of finite free $R$-modules of $k$. Using the extended derived tensor product on $D^-(\Mod(R))$, and writing $k$ as a homotopy colimit of cones of finite free $R$-modules, we see that this is indeed the case.
\end{proof}

In order to lift objects from $D^b(\ccc)$ to $D^b(\ddd)$, we need to impose the further condition of \emph{finite flat dimension}.
For $C \in D^-(\ccc)$, 
$$\mathrm{fd}(C) = \min \{ n \in \N \,\, | \,\, \forall M \in \mmod(k), \forall |i| > n \,\, \Tor_i^k(M, C) = 0\}$$
if such an $n$ exists and $\mathrm{fd}(C) = \infty$ otherwise.
We put $D^{-}_{\mathrm{ffd}}(\ccc) \subseteq D^-(\ccc)$ the full subcategory of objects with finite flat dimension. Clearly, $D^{-}_{\mathrm{ffd}}(\ccc)$ is a triangulated subcategory and $D^{-}_{\mathrm{ffd}}(\ccc) \subseteq D^b(\ccc)$.

\begin{proposition}\label{propfdlift}\label{propflatlift}
Consider $D \in D^-(\ddd)$ and suppose $C = k \otimes^L_R D \in D(\ccc)$ has $\mathrm{fd}(C) \leq n$. Then $\mathrm{fd}(D) \leq n$ too. In particular, if $C \cong H^0(C)$ and $H^0(C)$ is flat, then $D \cong H^0(D)$ and $H^0(D)$ is flat.
\end{proposition}

\begin{proof}
For any $X \in \mmod(k)$, we have $X \otimes^L_R D \cong X \otimes^L_k (k \otimes_R^L D) = X \otimes^L_k C$
so $\Tor_i^R(X,D) = H^i(X \otimes^L_R D) = 0$ for $|i| > n$. For an arbitrary $Y \in \mmod(R)$, the exact sequence $0 \lra IY \lra Y \lra k \otimes_R Y \lra 0$ easily yields that $\Tor_i^R(Y,D) = 0$ for  $|i| > n$.
\end{proof}

\subsection{Mutation and deformation}

Let $\ccc$ be a $k$-linear Grothendieck category. In this section we define mutations in the derived category $D(\ccc)$.
We will use the following standard concepts (see \cite{bridgelandstern}, \cite{gorodentsevkuleshov}):

\begin{enumerate}
\item An object $E \in D(\ccc)$ is \emph{exceptional} if $\RHom_{\ccc}(E, E) \cong k$;
\item A sequence of objects $E_0, E_1, \dots E_k$ is \emph{exceptional} if all the objects $E_i$ are exceptional and moreover $\RHom_{\ccc}(E_j, E_i) = 0$ for $j > i$.
\item A sequence of objects $E_0, E_1, \dots E_k$ is \emph{strong exceptional} if it is exceptional and moreover $\RHom_{\ccc}(E_i, E_j) \cong D(\ccc)(E_i, E_j)$ for all $i, j$.
\end{enumerate}

Consider $E, E_0, \dots E_k, C \in D^b(\ccc)$.
From Proposition \ref{propchange}, we obtain canonical morphisms $\RHom_{\ccc}(E,C) \otimes^L_k E \lra C$ and $C \lra \RHom_{\ccc}(C,E) \otimes^L_k E$ in $D(\ccc)$. 

\begin{enumerate}
\item The \emph{left mutation of $C$ through $E$} is defined by the triangle
$$\RHom_{\ccc}(E,C) \otimes^L_k E \lra C \lra L_E(C) \lra.$$
\item The \emph{left mutation of $C$ through $(E_0, \dots E_k)$} is
$$L_{(E_0, \dots E_k)}(C) = L_{E_0}L_{E_1}\dots L_{E_k}(C).$$
\item The \emph{right mutation of $C$ through $E$} is defined by the triangle
$$R_E(C) \lra C \lra \RHom_k(\RHom_{\ccc}(C,E), E) \lra.$$
\item The \emph{right mutation of $C$ through $(E_0, \dots E_k)$} is
$$R_{(E_0, \dots E_k)}(C) = R_{E_k}R_{E_{k-1}}\dots R_{E_0}(C).$$
\end{enumerate}

For a collection of objects $\eee \subseteq D^b(\ccc)$, we put
$$^{\perp}\eee = \{ C \in D^b(\ccc) \,\,| \,\, \RHom_{\ccc}(C, E) = 0 \,\, \forall \,\, E \in \eee\}$$
$$\eee^{\perp} = \{ C \in D^b(\ccc) \,\,| \,\, \RHom_{\ccc}(E, C) = 0 \,\, \forall \,\, E \in \eee\}$$

\begin{proposition}
Suppose the objects $E, E_0, \dots, E_k$ are exceptional and compact in $D(\ccc)$.
\begin{enumerate}
\item We obtain inverse equivalences $L_E: ^{\perp}\!\!E \lra E^{\perp}$ and $R_E: E^{\perp} \lra ^{\perp}\!\!E$.
\item We obtain inverse equivalences $L_{(E_0, \dots E_k)}: ^{\perp}\!\!(E_0, \dots, E_k) \lra (E_0, \dots, E_k)^{\perp}$ and $R_{(E_0, \dots E_k)}: (E_0, \dots, E_k)^{\perp} \lra  ^{\perp}\!\!(E_0, \dots, E_k)$.
\end{enumerate}
\end{proposition}

Following \cite{bridgelandstern}, we define helices depending on two positive integers.

\begin{definition}\label{defhel}
A sequence $H = (E_i)_{i \in \Z}$ in $D^b(\ccc)$ is an $(n,d)$-\emph{helix} (for positive integers $n$ and $d$ with $d \geq 2$) if:
\begin{enumerate}
\item For each $i \in \Z$ the corresponding \emph{thread} $(E_i, E_{i +1}, \dots, E_{i +n -1})$ is an exceptional collection of compact generators of $D(\ccc)$;
\item For each $i \in \Z$ $$E_{i - n} = L_{(E_{i - (n-1)}, \dots E_{i-1})}(E_i)[1-d].$$
\end{enumerate}
\end{definition}
A helix is called \emph{strong} if every thread is strong exceptional and \emph{geometric} if for all $i < j$, $\RHom_{\ccc}(E_i, E_j) \cong D(\ccc)(E_i, E_j)$. 

\begin{theorem}\label{helixlift}
Let $\ddd$ be a (flat) Grothendieck deformation of $\ccc$ and suppose $H = (E_i)_{i \in \Z}$ is an $(n,d)$-helix with $E_i \in D^-_{\mathrm{ffd}}(\ccc)$. 
\begin{enumerate}
\item There is a unique $(n,d)$-helix $\bar{H} = (\bar{E}_i)_{i \in \Z}$ with $\bar{E}_i \in D^-_{\mathrm{ffd}}(\ddd)$ and $k \otimes^L_R \bar{E}_i = E_i$ for all $i \in \Z$.
\item If $E_i$ is a flat object in $\ccc$, then $\bar{E}_i$ is a flat object in $\ddd$ with $k \otimes_R \bar{E}_i \cong E_i$.
\item If $H$ is strong and $D(\ccc)(E_i, E_j)$ is a flat $k$-module for $i < j$ and $j-i < n$, then $\bar{H}$ is strong and $D(\ddd)(\bar{E}_i, \bar{E}_j)$ is a flat $R$-module with $k \otimes_R D(\ddd)(\bar{E}_i, \bar{E}_j) \cong D(\ccc)(E_i, E_j)$ for $i < j$ and $j-i < n$.
\item If $H$ is geometric and $D(\ccc)(E_i, E_j)$ is a flat $k$-module for $i < j$, then $\bar{H}$ is geometric and $D(\ddd)(\bar{E}_i, \bar{E}_j)$ is a flat $R$-module with $k \otimes_R D(\ddd)(\bar{E}_i, \bar{E}_j) \cong D(\ccc)(E_i, E_j)$ for $i < j$.

\end{enumerate}
\end{theorem}

\begin{proof}
Since every $E_i$ is compact and exceptional, we have $\RHom_{\ccc}(E_i, I \otimes^L_k E_i) \cong I \otimes^L_k \RHom_{\ccc}(E_i, E_i) \cong I$ so according to \cite{lowen2}, there is a unique derived lift $\bar{E}_i \in D^-(\ddd)$ with $k \otimes^L_R \bar{E}_i \cong E_i$. By Proposition \ref{propfdlift}, the objects $\bar{E}_i$ have bounded flat dimension. By Propositions \ref{propcomplift} and \ref{propgenlift}, $(\bar{E}_i, \dots, \bar{E}_{i + n -1})$ is a collection of compact generators for $D(\ddd)$. By Proposition \ref{propcompchange} and adjunction, we have $$k \otimes^L_R \RHom_{\ddd}(\bar{E}_i, \bar{E}_j) = \RHom_{\ccc}(k \otimes_R^L E_i, k \otimes_R^L E_j) = \RHom_{\ccc}(E_i, E_j).$$
Looking at the abelian deformation $\Mod(R)$ of $\Mod(k)$, we then have the following facts. By derived Nakayama, $\RHom_{\ccc}(E_i, E_j) = 0$ implies $\RHom_{\ddd}(\bar{E}_i, \bar{E}_j) = 0$. If $\RHom_{\ccc}(E_i, E_j) \cong k$, then necessarily $\RHom_{\ddd}(\bar{E}_i, \bar{E}_j) = R$ so $(\bar{E}_i, \dots, \bar{E}_{i + n -1})$ is an exceptional collection. If $\RHom_{\ccc}(E_i, E_j)$ is isomorphic to the flat $k$-module $D(\ccc)(E_i, E_j)$, then by Proposition \ref{propflatlift} $\RHom_{\ddd}(\bar{E}_i, \bar{E}_j)$ is isomorphic to the flat $R$-module $D(\ddd)(\bar{E}_i, \bar{E}_j)$ and $k \otimes_R D(\ddd)(\bar{E}_i, \bar{E}_j) \cong D(\ccc)(E_i, E_j)$. In particular, strongness and geometricity of the helix lift. Finally, the helix condition (2) for $\bar{H}$ easily follows from Lemma \ref{mutdef}.
\end{proof}

\begin{lemma}\label{mutdef}
Let $\ddd$ be a flat Grothendieck deformation of $\ccc$. Consider $E, C \in D^b(\ddd)$ with $E$ compact. Then
$$k \otimes^L_R L_E(C) \cong L_{k \otimes_R^L E}(k \otimes_R^L C).$$
\end{lemma}

\begin{proof}
This easily follows from the following computation:
\begin{align*}
k \otimes^L_R [\RHom_{\ddd}(E, C) \otimes^L_R E] &= [k \otimes^L_R \RHom_{\ddd}(E, C)] \otimes^L_R E \\ & = \RHom_{\ccc}(k \otimes^L_R E, k \otimes^L_R C) \otimes^L_R E\\
& = \RHom_{\ccc}(k \otimes^L_R E, k \otimes^L_R C) \otimes^L_k (k \otimes^L_R E)
\end{align*}
where we have used Propositions \ref{propcompchange} and \ref{CR}.
\end{proof}

\subsection{$\Z$-algebras versus matrix algebras}
Let $\ccc$ be a Grothendieck category with a sequence of flat objects $(\ooo(n))_{n \in \Z}$ in $\ccc$. We are interested in the following situation:

\begin{enumerate}
\item $(\ooo(n))_{n \in \Z}$ is a $\Z$-generating sequence in $\ccc$;
\item $(\ooo(n))_{n \in \Z}$ is a geometric $(k,d)$-helix in $D(\ccc)$.
\end{enumerate}

\noindent In this situation, there are two natural associated algebraic objects:
\begin{enumerate}
\item The $\Z$-algebra $\AAA$ with
$$\AAA(n,m) = \begin{cases} \ccc(\ooo(-n), \ooo(-m)) & \text{if}\,\, n \geq m\\ 0 &\text{otherwise}\end{cases}$$
\item The full subcategory $\AAA_{[i-k, i]}$ of $\AAA$ spanned by the objects $i-k, \dots, i-1, i$.
\end{enumerate}
Under the listed conditions, we then have:
$$\ccc \cong \Qmod(\AAA) \hspace{1,5cm} \text{and} \hspace{1,5cm} D(\ccc) \cong D(\AAA_{[i-k,i]}).$$

\begin{theorem}\label{mainderdef}
There are equivalences of deformation functors
$$\xymatrix{ {\Def_{\mathrm{ab}}(\ccc)} \ar[r]_-{\lambda} & {\Def_{\mathrm{lin}}(\AAA)} \ar[r]_-{\rho} & {\Def_{\mathrm{lin}}(\AAA_{[i-k, \dots i]}).}}$$
\begin{enumerate}
\item For $\ddd \in \Def_{\mathrm{ab}}(\ccc)$, let $(B(n))_{n \in \Z}$ be the sequence of the unique flat lifts $B(n)$ of $\ooo(n)$ along $k \otimes_R -$. 
Then $(B(n))_{n \in \Z}$ satisfies conditions (1) and (2) and $\lambda(\ddd)$ is the $\Z$-algebra $\BBB$ associated to this sequence. In particular, $(B(n))_{n \in \Z}$ is the unique helix in $D(\ddd)$ with $k \otimes_R^L B(n) \cong \ooo(n)$.
\item For $\BBB \in \Def_{\mathrm{lin}}(\AAA)$, $\rho(\BBB) = \BBB_{[i-k, \dots i]}$, the full subcategory of $\BBB$ spanned by the objects $i-k, \dots, i-1, i$.
\end{enumerate}
\end{theorem}

\begin{proof}
We already know from Theorem \ref{maindef} that $\lambda$ defines an equivalence of deformation functors, and that $(B(n))_{n \in \Z}$ satisfies condition (1). By Theorem \ref{helixlift}, $(B(n))_{n \in \Z}$ is the unique helix in $D(\ddd)$ with $k \otimes_R^L B(n) \cong \ooo(n)$, and it is a geometric helix.

In order to show that $\rho$ is an equivalence too, we will construct an inverse equivalence $\kappa$.  
We start with a flat linear deformation $\BBB_{[i-k, \dots i]}$ of $\AAA_{[i-k, \dots i]}$.
Consider the induced flat abelian deformation $\Mod(\AAA_{[i-k, i]}) \lra \Mod(\BBB_{[i-k, i]})$ and let $H = (A(n))_{n \in \Z}$ be the $(k,d)$-helix spanned by $A(-i), \dots A(-i + k)$ in $D(\AAA_{[i-k, \dots i]})$. The objects $A(n)$ all have bounded flat dimension. By the derived equivalence $D(\ccc) \cong D(\AAA_{[i-k,i]})$, $H$ is a geometric helix and $D(\AAA_{[i-k, \dots i]})(A(n), A(m))$ is flat for $n \leq m$. From Theorem \ref{helixlift}, we thus obtain a unique $(n,d)$-helix $\bar{H} = (B(n))_{n \in \Z}$ in $D(\BBB_{[i-k, \dots i]})$ which is geometric and such that $D(\BBB_{[i-k, \dots i]})(B(n), B(m))$ is a flat $R$-module with $k \otimes_R D(\BBB_{[i-k, \dots i]})(B(n), B(m)) \cong D(\AAA_{[i-k, \dots i]})(A(n), A(m))$ for $n \leq m$. We now define the $\Z$-algebra $\BBB$ with
$$\BBB(n,m) = \begin{cases} D(\AAA_{[i-k, \dots i]})(B(-n), B(-m)) & \text{if}\,\, n \geq m\\ 0 & \text{otherwise} \end{cases}$$
and composition inherited from $D(\AAA_{[i-k, \dots i]})$. 
Then $\BBB$ is a $\Z$-algebra deforming $\AAA$, and we put $\kappa(\BBB_{[i-k, \dots i]}) = \BBB$. 

Next we are to verify that $\kappa$ and $\rho$ are inverse equivalences. It is clear that restricting the $\Z$-algebra $\BBB$ we just constructed to the objects $i -k, \dots i$ yields back (an isomorphic copy of) the original $\BBB_{[i-k, \dots i]}$, since the representable modules in $\Mod(\AAA_{[i-k, i]})$ lift precisely to the corresponding representable modules in $\Mod(\BBB_{[i-k, i]})$. 

For the other direction, we may - because of the equivalence given by $\lambda$ -  start with a $\Z$-algebra deformation $\BBB$ of $\AAA$ obtained from an abelian deformation $\ddd$ of $\ccc$ by lifting the flat objects $\ooo(n) \in \ccc$ to flat objects $E(n) \in \ddd$.
By Theorem \ref{helixlift}, $(E(n))_{n \in \Z}$ is the unique helix in $D(\ddd)$ lifting $(\ooo(n))_{n \in \Z}$ in $D(\ccc)$. In particular, for the restriction $\BBB_{[i-k, \dots, i]}$, we obtain an equivalence $D(\ddd) \cong D(\BBB_{[i-k, \dots, i]})$. If we now use $D(\BBB_{[i-k, \dots, i]})$ to construct $\kappa(\BBB_{[i-k, \dots i]})$, then this equivalence of categories yields the required isomorphism $\BBB \cong \kappa(\BBB_{[i-k, \dots i]})$.
\end{proof}

\def\cprime{$'$} \def\cprime{$'$}
\providecommand{\bysame}{\leavevmode\hbox to3em{\hrulefill}\thinspace}
\providecommand{\MR}{\relax\ifhmode\unskip\space\fi MR }
\providecommand{\MRhref}[2]{%
  \href{http://www.ams.org/mathscinet-getitem?mr=#1}{#2}
}
\providecommand{\href}[2]{#2}

\end{document}